\scrollmode
\documentclass[11pt]{amsart}

\usepackage{amssymb,amsmath,amsthm}
\usepackage{amsrefs}
\usepackage[all,arc]{xy}
\usepackage[nointegrals]{wasysym}
\usepackage{graphicx}
\usepackage{mathrsfs}
\usepackage{tikz}
\usepackage{tikz-cd}

\usepackage{multicol}
\usepackage{fancyhdr}
\usepackage{cite}
\usepackage[colorlinks=true, citecolor=red, linkcolor=red, urlcolor=magenta]{hyperref}
\usepackage{float}
\usepackage{listings}

\title{Subgroups of Bestvina-Brady groups}

\DeclareMathOperator{\cd}{cd}

\DeclareFontFamily{U}{wncy}{}
\DeclareFontShape{U}{wncy}{m}{n}{<->wncyr10}{}
\DeclareSymbolFont{mcy}{U}{wncy}{m}{n}
\DeclareMathSymbol{\Sha}{\mathord}{mcy}{"58}
\DeclareMathSymbol{\sha}{\mathord}{mcy}{"78}

\newcommand{\K}{\mathbb{K}}

\newcommand{\F}{\mathbb{F}}

\newcommand{\Z}{\mathbb{Z}}

\newcommand{\mc}[1]{\mathscr{#1}}
\newcommand{\bu}{\bullet}

\newcommand{\pres}[2]{\left\langle{#1}\, \big\vert\, {#2}\right\rangle}

\newcommand{\li}{{\mc L}}
\newcommand{\bbu}{{\bu,\bu}}
\newcommand{\ext}{\operatorname{Ext}}
\newcommand{\argu}{\hbox to 1.5ex{\hrulefill}}

\newcommand{\gr}{\operatorname{gr}}
\newcommand{\gen}[1]{\langle{#1}\rangle}
\newcommand{\mf}[1]{\mathfrak{#1}}
\newcommand{\abs}[1]{\vert{#1}\vert}

\newcommand{\hnn}{\operatorname{HNN}}

\usepackage{geometry}
\author{S. Blumer}
\address{Fakultät für Mathematik, Universität Wien, Oskar-Morgenstern-Platz 1, 1090
	Wien, Austria}
\email{simone.blumer@univie.ac.at}

\date{\today}

\begin{document}
{	\maketitle
	
	
	
	\newtheorem{thm}{Theorem}[section]
	\newtheorem*{thmA}{Theorem A}
	\newtheorem*{thmB}{Theorem B}
	\newtheorem*{thm*}{Theorem}
	\newtheorem{cor}[thm]{Corollary}
	\newtheorem{lem}[thm]{Lemma}
	\newtheorem{prop}[thm]{Proposition}
	\newtheorem{defin}[thm]{Definition}
	\theoremstyle{definition}
	\newtheorem{exam}[thm]{Example}
	\theoremstyle{definition}
	\newtheorem{examples}[thm]{Examples}
	\newtheorem{rem}[thm]{Remark}
	\newtheorem{case}{\sl Case}
	
	\newtheorem{claim}{Claim}
	\newtheorem{fact}[thm]{Fact}
	\newtheorem{question}[thm]{Question}
	\newtheorem*{questionss}{Questions}
	\newtheorem{conj}[thm]{Conjecture}
	\newtheorem*{notation}{Notation}
	\swapnumbers
	\newtheorem{rems}[thm]{Remarks}
	
	\theoremstyle{definition}
	\newtheorem*{acknowledgment}{Acknowledgment}

	\numberwithin{equation}{section}

	}

	\begin{abstract}
		In \cite{droms}, Droms proved that all the subgroups of a right-angled Artin group (RAAG) defined by a finite simplicial graph $\Gamma$ are themselves RAAGs if, and only if, $\Gamma$ has no induced square graph nor line-graph of length $3$. The present work provides a similar result for specific normal subgroups of RAAGs, called Bestvina-Brady groups: We characterize those graphs in which every subgroup of such a group is itself a RAAG. In turn, we confirm several Galois theoretic conjectures for the pro-$p$ completions of these groups.
	\end{abstract}
	
	\section*{Introduction}
Let $\Gamma$ be a simplicial graph, i.e., a pair $(V,E)$, where $V$ is a set and $E$ is a set of $2$-element subsets of $V$. The associated \textbf{right-angled Artin group} $G_\Gamma$ (RAAG, for short), or graph group, is the group generated by the vertices of the graph, with a defining relation $vw=wv$ for each pair of vertices $v,w$ joined by an edge in $\Gamma$:
 \begin{align}\label{eq:RAAG}
	\pres{v:\ v\in V}{[v,w]:\ \{v,w\}\in E}
\end{align} 
A group $G$ is said to be a RAAG if it is isomorphic to some $G_\Gamma$. In that case, $\Gamma$ is determined by $G$ (see \hspace{1pt}\cite{kimroush}), and is called the defining graph of $G$. RAAGs have intensively been studied for they provide a rich class of groups whose algebraic properties can be derived from those of the defining graph. 
Such groups sit in between the two extremal cases of free groups and free abelian groups, and they can thus be seen as interpolating those classes. It is then natural to ask which properties RAAGs share with free/free abelian groups.

	For instance, it is well known that subgroups of free and free abelian groups are of the same type.
	Nevertheless, the subgroups of a RAAG might not be RAAGs themselves: depending on the defining graph, there may exist finitely generated subgroups that do not even admit a finite number of defining relation, that is to say, the RAAG is not coherent (see Prop. \ref{prop:cohe}). 
	
	This problem was studied by Droms \cite{dromschordal}, who proved a finitely generated RAAG is coherent if, and only if, the defining graph is \textbf{chordal}, i.e., it does not contain any induced cycle of length $\geq 4$.
	
	On the other hand, Droms \cite{droms} also gave a complete classification of finite graphs $\Gamma$ such that all subgroups of $G_\Gamma$ are RAAGs. 
	\begin{thm}[Droms]\label{thm:droms}
		Let $\Gamma$ be a finite graph. Then, every subgroup of $G_\Gamma$ is a RAAG if, and only if, no induced subgraph of $\Gamma$  is either a square or a line-graph of length $3$.
	\end{thm} 
	Graphs without induced squares nor line-graphs of length $3$ appear in various contexts and with several names, as trivially perfect graphs \cite{triviallyperfect}, graphs of elementary type \cite{bqw}, or quasi-threshold graphs \cite{quasithreshold}. In honour to Droms, in the context of RAAGs, these graphs are usually called \textbf{Droms graphs}, and we will adopt the same denomination. Clearly, Droms graphs are chordal.
	
	It follows from the proof of Theorem \ref{thm:droms} that, for a finite graph $\Gamma$ to be a Droms graph, it is also sufficient that all the \textit{finitely generated} subgroups of $G_\Gamma$ are RAAGs, i.e., $G_\Gamma$ is locally RAAG.
	
	Among the subgroups of RAAGs, some distinguished normal subgroups have been studied since their discovery due to Bestvina and Brady in \cite{bb}, as they provide examples of finitely generated, yet non-finitely presented groups (in fact, not even of type $F_2$), that are of type FP.
Given a graph $\Gamma$, one can define the corresponding \textbf{Bestvina-Brady group} as the kernel of the ``length character" $\chi_\Gamma:G_\Gamma\to \Z$, obtained by extending $v\in V\mapsto 1\in \Z$.
	
	If $\Gamma$ is a Droms graph, then clearly $B_\Gamma$ is itself a RAAG, but the problem of detecting Bestvina-Brady groups that are RAAGs remains open in general.
	 Moved by Droms' Theorem \ref{thm:droms}, we are interested in understanding when the group $B_\Gamma$ is locally RAAG: clearly, if $G_\Gamma$ is locally RAAG, then so is $B_\Gamma$, but one can easily construct counterexamples to the converse statement (see Example \ref{ex:counter}(1)). 
	 In particular, we show
	\begin{thmA}\label{thmA}
		Let $\Gamma$ be a finite connected simplicial graph. Then, every subgroup of the Bestvina-Brady group $B_\Gamma$ is a RAAG if, and only if, $\Gamma$ is {\rm chordal} and no induced subgraph of $\Gamma$ has either of the two forms: 
	
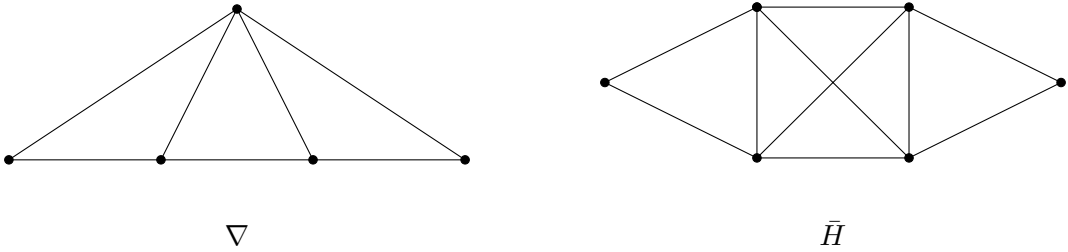
\begin{figure}[H]
		\centering
		\begin{tikzpicture}
			\draw (0,0) -- (2,0) -- (4,0) -- (6,0);
			\draw (0,0) -- (3,2);
			\draw (2,0) -- (3,2);
			\draw (4,0) -- (3,2);
			\draw (6,0) -- (3,2);
			\node [fill=black!100,circle,scale=0.3,draw] at (0,0) {};
			\node [fill=black!100,circle,scale=0.3,draw] at (2,0) {};
			\node [fill=black!100,circle,scale=0.3,draw] at (4,0) {};
			\node [fill=black!100,circle,scale=0.3,draw] at (6,0) {};
			\node [fill=black!100,circle,scale=0.3,draw] at (3,2) {};
			\node at (3,-1) {$\nabla$};
		\end{tikzpicture}
		~\qquad\qquad
		\begin{tikzpicture}
			\draw (-2,1) -- (0,2) -- (2,2) -- (4,1) -- (2,0) -- (0,0) -- (-2,1);
			\draw (0,0) -- (0,2) -- (2,0) -- (2,2) -- (0,0);
			
			\node [fill=black!100,circle,scale=0.3,draw] at (0,0) {};
			\node [fill=black!100,circle,scale=0.3,draw] at (2,0) {};
			\node [fill=black!100,circle,scale=0.3,draw] at (2,2) {};
			\node [fill=black!100,circle,scale=0.3,draw] at (0,2) {};
			\node [fill=black!100,circle,scale=0.3,draw] at (0,2) {};
			\node [fill=black!100,circle,scale=0.3,draw] at (-2,1) {};
			\node [fill=black!100,circle,scale=0.3,draw] at (4,1) {};
			\node at (1,-1) {$\bar H$};
		\end{tikzpicture}
		
	\end{figure}
	\end{thmA}
	The graph $\nabla$ is called the \textbf{gem graph}, and $\bar H$ is a graph where two gems are overlapped; we will call it an \textbf{overlapping-gems graph}. 
	
	The pro-$p$ completions of right-angled Artin groups (and some generalizations of them) have been studied in \cite{ilirPavel} (and \cite{bqw}) in the attempt of confirming several Galois theoretic conjectures on the realizability of a pro-$p$ group as an absolute Galois group. In particular, 
	\begin{thm}[Snop\v ce, Zalesskii \cite{ilirPavel}]
		A finite graph $\Gamma$ is a Droms graph if, and only if, the pro-$p$ RAAG $G_{\Gamma,p}$ is isomorphic to the maximal pro-$p$ quotient of the absolute Galois group of a field containing a primitive $p^{th}$ root of $1$.
	\end{thm}
	
	We deduce an analogue of that result for the pro-$p$ completion of the Bestvina-Brady groups.
	
	Following \cite{weig}, \cite{sb_kosz}, and \cite{cmp}, we also study the Lie algebra counterpart ($\li_\Gamma$, and $\mc B_\Gamma$) of RAAGs and Bestvina-Brady groups. The advantage of working with positively graded Lie algebras relies on the fact that cohomology computations are easier than in the group-case. For instance, by means of a spectral sequence due to May, the cohomology ring of Bestvina-Brady groups defined on acyclic flag complexes is computed.
	\begin{prop}
		 If $\Gamma$ is a finite graph with acyclic flag complex over a field $k$, then  \[H^\bu(B_\Gamma,k)\simeq H^\bu(G_\Gamma,k)/(\chi_\Gamma)\] where $(\chi_\Gamma)$ denotes the ideal of $H^\bu(G_\Gamma,k)$ generated by the class of the $1$-cocycle $\chi_\Gamma:G\to k$.
	\end{prop} 
	
	In \cite{servdroms}, the authors produce embeddings of non-abelian surface groups into (the derived subgroup of) RAAGs defined by graphs with induced cycles of length at least $4$. We show that no non-abelian surface group appears as a Bestvina-Brady group. We also prove a similar result for non-abelian Poincaré duality Lie algebras.
	
	At the end of the paper we show that the Bestvina-Brady group of a graph is coherent precisely when so is the right-angled Artin group, that is, when the graph is chordal.

	\begin{acknowledgment}
	The author is supported by the Austrian Science Foundation (FWF), Grant DOI 10.55776/P33811.
	\end{acknowledgment}

	\section{Right-angled Artin \& Bestvina-Brady groups and Lie algebras}
	Since we are interested both in right-angled Artin groups and their Lie algebraic counterparts, we will work in a more general setting.
	Henceforth, let $\mathcal A$ be either the category of ($p$-restricted) positively-graded Lie algebras over a commutative ring (resp. over the finite field $\F_p$, $p$ odd), or the category of (pro-$p$) groups. 
	\subsection{RAAGs}If $\Gamma$ is a simplicial graph, we denote by $\mf R_\Gamma$ the \textbf{RAAG object} given by the presentation (\ref{eq:RAAG}) in the category $\mathcal A$, where the generators have degree $1$ in the Lie algebra case. We will identify the canonical generators of $\mf R_\Gamma$ with the vertices of $\Gamma$. 

Many algebraic properties of $\mf R_\Gamma$ (e.g., the form of the trivial coefficient cohomology, the coherence property, and the decomposability into free/direct product) do not depend on the chosen category, but only on the underlying graph. 

Nevertheless, Theorem \ref{thm:droms} is not true in the case of RAAG Lie algebras, as the following example shows.
	\begin{exam}\label{ex:2rel} Consider the graph $\Gamma$ with geometric realization
		
		
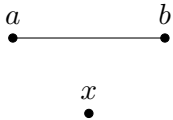
\begin{figure}[H]
			\centering
			\begin{tikzpicture}
				\draw (0,0) -- (2,0);
				
				\node [fill=black!100,circle,scale=0.3,draw,"$a$"] at (0,0) {};
				\node [fill=black!100,circle,scale=0.3,draw, "$b$"] at (2,0) {};
				\node [fill=black!100,circle,scale=0.3,draw,"$x$"] at (1,-1) {};
			\end{tikzpicture}
		\end{figure}

		Then, $\Gamma$ is a Droms graph and hence every (finitely generated) subgroup of $G_\Gamma$ is itself a RAAG. 
		The RAAG Lie algebra 
		$\li_\Gamma$ over an arbitrary field $k$ can be given a presentation $\pres{a,b,x}{[a,b]}$. The subalgebra 
		$\mc M$ generated by the elements $a,b, z=[x,a]$ and $t=[x,b]$ has minimal presentation \[\pres{a,b,z=[x,a],t=[x,b]}{[a,b],[z,b]+[t,a]},\] and hence it is not a RAAG Lie algebra (see \cite[Sec. 6]{cmp}). Notice that $\mc M$ is not contained in $\li_\Gamma'$, which is free by Proposition \ref{prop:cohe}. 
	\end{exam}
	However, by \cite{sb_kosz}, Theorem \ref{thm:droms} remains true if we restrict to the \textit{standard} subalgebras the RAAG Lie algebra defined by a Droms graph. Recall that a graded Lie algebra $\li=\bigoplus_{i=1}^\infty \li_i$ is said to be standard if it is generated by its elements of degree $1$, i.e., $[\li,\li]=\bigoplus_{i\geq 2}\li_i$. Such a Lie algebra is Bloch-Kato (or BK) if every standard subalgebra of $\li$ is quadratic, i.e., it admits a set of relations of degree $2$. Similarly, a pro-$p$ group $G$ is BK if the Galois cohomology ring $H^\bu(H,\F_p)$ is a quadratic algebra for every closed subgroup $H$ of $G$. For a comprehensive exposition on Koszul and BK Lie algebra, we refer the reader to the author's paper \cite{sb_kosz}.
	
		\begin{thm}\label{thm:bkraags}
		Let $\Gamma$ be a finite simplicial graph. Then, $\Gamma$ is a Droms graph if, and only if, anyone of the following statements holds:
		\begin{enumerate}
			\item All subgroups of $G_\Gamma$ are RAAGs (possibly defined by infinite graphs) (\hspace{1pt}\cite{droms}),
			\item All standard subalgebras of $\li_\Gamma$ are RAAG Lie algebras (\hspace{1pt}\cite[Prop. 3.17 ]{sb_kosz}),
			\item The RAAG Lie algebra $\li_\Gamma$ is BK (\hspace{1pt}\cite{sb}),
			\item The cohomology $k$-algebra $H^\bu(\mf R_\Gamma,k)$ is universally Koszul, where $k$ is a field (see \cite{cassquad},\cite{ilirPavel},\cite{sb}),
			\item The pro-$p$ RAAG $ G_{\Gamma,p}$ is BK (\hspace{1pt}\cite{ilirPavel}),
			\item Every closed subgroup of $ G_{\Gamma,p}$ is a pro-$p$ RAAG (\hspace{1pt}\cite{ilirPavel}),
			\item The pro-$p$ group $ G_{\Gamma,p}$ is isomorphic to the maximal pro-$p$ quotient $G_\K(p)$ of the absolute Galois groups of a field $\K$ containing a primitive $p^{th}$ root of $1$ (\hspace{1pt}\cite{ilirPavel}).
		\end{enumerate}
	\end{thm}

	An object $\mf C$ of $\mathcal A$ is called a Droms object if it is isomorphic to $\mf R_\Gamma$ for a Droms graph $\Gamma$, i.e., it is a RAAG object whose defining graph is a Droms graph.


	\subsection{Bestvina-Brady}

	 In this work, $k$ will always denote a ring whose nature depends on the chosen category.
	 \begin{enumerate}
	 	\item For Lie algebras, $k$ is an arbitrary field,
	 	\item For $p$-restricted Lie algebras, $k=\F_p$ is the finite field of order $p$,
	 	\item For abstract groups, $k=\Z$ is the ring of integers,
	 	\item For pro-$p$ groups, $k=\Z_p$ is the ring of $p$-adic integers.
	 \end{enumerate}
	 
	There exists a natural \textbf{length character} $\chi_\Gamma:\mf R_\Gamma\to k$ sending each vertex to $1\in k$. The kernel $\mf B_\Gamma$ of $\chi_\Gamma$ is the (right-angled)  Bestvina-Brady object. 
		If $\Lambda$ is induced in a graph $\Gamma$, since the character $\chi_\Lambda$ is the restriction of $\chi_\Gamma$ on the subalgebra $\mf R_\Lambda$, the object $\mf B_\Lambda$ naturally embeds into $\mf B_\Gamma$, and $\mf B_\Lambda=\mf R_\Lambda\cap \mf B_\Gamma$. 
		
		We will denote the (pro-$p$) groups by the letters $G$, $B$, ... (resp. $G_p$, $B_p$, ...), and the ($p$-restricted) graded Lie algebras by $\li,\mc M,\mc B$, ... For instance, the right-angled Artin (pro-$p$) group is denoted by $G_\Gamma$ (resp. $G_{\Gamma,p}$) and its Bestvina-Brady (pro-$p$) group by $B_\Gamma$ (resp. $ B_{\Gamma,p}$). Similarly, $\li_\Gamma$ and $\mc B_\Gamma$ (resp. $\li_{\Gamma,p}$ and $\mc B_{\Gamma,p}$) will denote the RAAG and the Bestvina-Brady (resp. $p$-restricted) Lie algebra of $\Gamma$.
		
		If $p$ is an odd prime, then $\li_{\Gamma,p}$ is the $p$-restrictification of the RAAG Lie $\F_p$-algebra $\li_\Gamma$ in the sense of \cite{sb_kosz} (see \cite{barth}), and hence their cohomology theories are the same; thus, we will only focus on the non-restricted case.
		
		In order to avoid the difference in the behavior of groups and Lie algebras as in Example \ref{ex:2rel}, we give the following definition of a \textit{local} property.
		If $\mathcal P$ is a group  theoretic (resp. Lie algebra theoretic) property, we say that a group $G$ (resp. a graded Lie algebra $\li$) is \textbf{locally} $\mathcal P$ if every finitely generated subgroups of $G$ (resp. every \textit{standard} subalgebra of $\li$) satisfies the property $\mathcal P$. By Theorem \ref{thm:bkraags}, Droms objects are locally Droms.

		In \cite{cmp2} and \cite{bb} the authors prove that the topology of the flag complex of a graph determines cohomological finiteness properties of the corresponding Bestvina-Brady objects. 
			\begin{thm}\label{thm:BB_FPn}
			If $n$ is a natural number, then the following statements are equivalent: 
			\begin{enumerate}
				\item The flag complex $\Delta_\Gamma$ on $\Gamma$ is $(n-1)$-acyclic over $k$, i.e., the reduced simplicial homology groups $\tilde H_i(\Delta_\Gamma,k)$ vanish over $k$ for $i\leq n-1$;
				\item The object $\mf B_\Gamma$ is of type FP$_n$ over $k$.
			\end{enumerate}
		\end{thm}
		
		In particular, by \cite[Cor. 2.7]{sb_kosz}, the Lie $k$-algebra $\mc B_\Gamma$ is Koszul if, and only if, $\Delta_\Gamma$ is acyclic over $k$.
		
		\begin{cor}
			Let $\Gamma$ be a finite simplicial graph. If $\Delta_\Gamma$ is acyclic over $k$, then, \[H^\bu(\mc B_\Gamma,k)\simeq H^\bu(\li_\Gamma,k)/(\chi_\Gamma\cdot H^1(\li_\Gamma,k)),\] where $\chi_\Gamma:\li_\Gamma\to k$ is the length character of $\li_\Gamma=\li_\Gamma(k)$, and $\mc B_\Gamma=\ker\chi_\Gamma$. Moreover, $H^\bu(\mc B_\Gamma,k)$ is a Koszul algebra. 
		\end{cor}
		
		Since the quadratic dual of the quadratic cover of $\mc B_\Gamma$ is isomorphic to $H^\bu(\li_\Gamma,k)/(\chi_\Gamma\cdot H^1(\li_\Gamma,k))=:A_\Gamma$ (see \cite{sb_kosz}), it	follows from \cite[Ch. 2, Cor. 3.3.]{pp} that $A_\Gamma$ is Koszul when so is $\mc B_\Gamma$. In particular, since $H^\bu(\li_\Gamma,k)$ is the Stanley-Reisner ring of the opposite graph of $\Gamma$, we get
		\begin{cor}\label{cor:coho}
			Let $\Gamma$ be a finite graph whose flag complex $\Delta_\Gamma$ is acyclic over $k$. 
			Then, quotient $A_\Gamma$ of the exterior $k$-algebra generated by the vertices of a finite simplicial graph $\Gamma$ by the ideal generated by the elements $\sum_{v\in V(\Gamma)}v$ and $x\wedge y$, for $\{x,y\}\notin E(\Gamma)$, is Koszul.
		\end{cor}

	In their seminal work \cite{bb}, Bestvina and Brady characterized the finite presentability of $B_\Gamma$ in terms of the homotopy type of the flag complex of the underlying graph. 
	
	\begin{thm}\label{thm:simply}
		Let $\Gamma$ be a finite graph. Then, $B_\Gamma$ is finitely presented if, and only if, the flag complex of $\Gamma$ is simply connected. 
	\end{thm}
	Notice that the same result does not hold for Bestvina-Brady Lie algebras, as, for graded Lie algebras over a field, being of type FP$_2$ and being finitely presented are equivalent conditions (see \cite{weig}). In fact, by Theorem \ref{thm:BB_FPn}, $\mc B_\Gamma$ is finitely presented if, and only if, the first homology group over $k$ of the flag complex of $\Gamma$ vanishes, the latter being a weaker condition than simple connectedness.
		
		\subsection{Lower central series}

For a group $G$, denote by $\gamma_n(G)$ the lower central series of $G$, i.e., $\gamma_1(G)=G$ and $\gamma_{n+1}(G)=[G,\gamma_n(G)]$. 
The associated graded object \[\gr^\gamma G=\bigoplus_{n\geq 1}\gamma_n(G)/\gamma_{n+1}(G)\] is a Lie ring, where the Lie brackets are induced by the commutator map $(g,h)\mapsto [g,h]=ghg^{-1}h^{-1}$.
 Similarly, if $G$ is a pro-$p$ group, then $\gr^\zeta G$ is computed in terms of the Jennings-Zassenhaus series, and is a $p$-restricted Lie algebra.
Since the filtration is clear from the context, we will drop the superscript and denote both these graded objects by $\gr G$.

As an example, if $\Gamma$ is a finite graph, then $\gr G_\Gamma$ is isomorphic with the RAAG Lie ring $\li_\Gamma(\Z)$ (see \cite{psRAAG}), and $\gr G_{\Gamma,p}\simeq \li_{\Gamma,p}$ (see \cite{barth}). 
In the same way, Bestvina-Brady groups and Lie algebras are related by the following result, which is an easy consequence of \cite[Thm. 5.6]{psBB}.
		\begin{lem}
			Let $\Gamma$ be a finite simplicial graph. If $\Gamma$ is connected, then $\gr B_\Gamma$ is naturally isomorphic with the Lie ring $\mc B_\Gamma(\Z)$. Similarly, $\gr B_{\Gamma,p}\simeq \mc B_{\Gamma,p}$.
		\end{lem}
		
		It follows that, for a connected graph $\Gamma$ and a field $k$, the $k$-cohomology ring of the (pro-$p$) RAAG $B_\Gamma$ can be computed in terms of the bigraded algebra $H^\bbu(\mc B_\Gamma,k)$ via a distinguished spectral sequence discovered by J.P. May (see \cite{may}). For instance, if $\mc B_\Gamma$ is a Koszul Lie algebra over a field $k$, then we get a $k$-algebra isomorphism $H^\bu(B_\Gamma,k)\simeq H^\bu(\mc B_\Gamma,k)$, proving that $B_\Gamma$ has Koszul cohomology over $k$.

		\begin{examples}\label{ex:counter}
		(1) Since the gem graph $\nabla$ is a cone on the line graph $L_3$ of length $3$, we get a decomposition $\mf R_\nabla=\mf R_{L_3}\times k$. It also follows that $\mf B_\nabla=\mf R_{L_3}$, and hence $\mf B_\nabla$ does not satisfy (1)--(7) of Theorem \ref{thm:bkraags}. 
		
		(2) Since the flag complex of the overlapping-gems graph $\bar H$ is contractible, by Theorems \ref{thm:BB_FPn} and \ref{thm:simply}, $\mf B_{\bar H}$ is finitely presented. We get the presentation (see \cite[Cor. 2.3]{psBB})
		\[\mf B_{\bar H}=\pres{e_1,e_2,e_3,e_4,e_5}{[e_1,e_2],[e_2,e_3],[e_2,e_4],[e_3,e_4],[e_4,e_5]}\]
		In particular, $\mf B_{\bar H}$ is a RAAG, with underlying graph that is not Droms (Figure \ref{fig:overgembb}). 
		 \end{examples}

		 	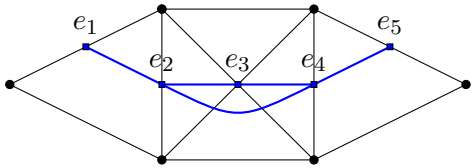
\begin{figure}[H]
		 	\centering
		 		\begin{tikzpicture}
		 		\draw (-2,1) -- (0,2) -- (2,2) -- (4,1) -- (2,0) -- (0,0) -- (-2,1);
		 		
		 		\draw (0,0) -- (0,2) -- (2,0) -- (2,2) -- (0,0);
		 		\draw[blue, thick] (-1,1.5) -- (0,1) -- (1,1) -- (2,1) -- (3,1.5);
		 		\draw[blue,thick] (0,1) .. controls (1,0.5) .. (2,1);

		 		\node [fill=black!100,circle,scale=0.3,draw] at (0,0) {};
		 		\node [fill=black!100,circle,scale=0.3,draw] at (2,0) {};
		 		\node [fill=black!100,circle,scale=0.3,draw] at (2,2) {};
		 		\node [fill=black!100,circle,scale=0.3,draw] at (0,2) {};
		 		\node [fill=black!100,circle,scale=0.3,draw] at (0,2) {};
		 		\node [fill=black!100,circle,scale=0.3,draw] at (-2,1) {};
		 		\node [fill=black!100,circle,scale=0.3,draw] at (4,1) {};
		 		\node[fill=blue!,rectangle,scale=0.3,draw,"$e_1$"] at (-1,1.5) {};
		 		\node[fill=blue!,rectangle,scale=0.3,draw,"$e_2$"] at (0,1) {};
		 		\node[fill=blue!,rectangle,scale=0.3,draw,"$e_3$"] at (1,1) {};
		 		\node[fill=blue!,rectangle,scale=0.3,draw,"$e_4$"] at (2,1) {};
		 		\node[fill=blue!,rectangle,scale=0.3,draw,"$e_5$"] at (3,1.5) {};
		 	\end{tikzpicture}
		 	\caption{The overlapping-gems graph $\bar H$ (in black) and the defining graph of $\mf B_{\bar H}$ (in blue).}\label{fig:overgembb}
		 \end{figure}

	\section{Graph characterization}
	Many classes of graphs are defined in terms of forbidden induced subgraphs. For instance, a finite simplicial graph $\Gamma$ is chordal if it contains no induced $n$-cycle for $n\geq 4$, i.e., chordal graphs are $(C_n)_{n\geq 4}$-free. Similarly, $\Gamma$ is a Droms graph if it does not contain any induced square nor line $L_3$ of length $3$: it is $(L_3,C_4)$-free. Sometimes,
	a family $C$ of graphs with forbidden induced subgraphs is closed with respect to some non-trivial operations. This means that $C$ has a subset $\Sigma$ of \textit{building-blocks} that generates $C$ by iterating those operations. If $\Sigma\neq C$, then we say that the family $C$ has a \textbf{defining construction} in terms of operations and building blocks $\Sigma$.
	For example, Droms graphs  can be obtained by applying cones and disjoint unions on single vertices \cite{droms}. Another class of graphs admitting a defining construction is that of ptolemaic graphs. 
	
	\subsection{Ptolemaic graphs}

	A graph $\Gamma$ is said to be \textbf{ptolemaic} if it is connected, chordal and it does not contain any induced gem $\nabla$. 
	Ptolemaic graphs form an interesting and well studied class of graph for they are distance-hereditary, as it was proved in \hspace{1pt}\cite{ptol}. The same work characterizes ptolemaic graphs in terms of a defining construction that we recall here.
	
	The operations are given by adding either a leaf or \textit{any kind} of twin, with the exception that a false-twin can only be attached to a vertex that has complete neighbourhood.
	
	\begin{enumerate}
		\item A \textbf{leaf} is a vertex that is adjacent to a single other vertex of the graph. 
		\item A (true) \textbf{twin} is a vertex that has the same neighbourhood as that of another vertex. 
		\item A \textbf{false-twin} is a vertex that shares the neighbourhood with another vertex to which it is not adjacent. 
	\end{enumerate}
	
	\begin{lem}[Bandelt, Mulder \hspace{1pt}\cite{ptol}]
		Ptolemaic graphs have a defining construction given by the above operations and single vertices as building blocks, i.e., every finite ptolemaic graph is obtained from the single vertex $K_1$ by iterating the operations of adding leaves or twins.
	\end{lem}
	
	For instance, block-graphs are ptolemaic graphs; recall that a block-graph is a graph obtained by connecting complete graphs at single common vertices. 
	
We now introduce a generalization of block-graphs, where the building blocks are the connected Droms graphs instead of cliques.
		A \textbf{tree of Droms graphs} is a graph defined by connected sums at vertices of connected Droms graphs.
		In particular, a tree of Droms graphs has a central element iff it is a Droms graph.
		
			\begin{defin}
			Let $\Gamma$ be a connected graph. A vertex $v$ is a \textbf{cut-vertex} for $\Gamma$ if the induced subgraph $\Gamma_v$ of $\Gamma$ obtained by removing $v$ is disconnected. A \textbf{block of $\Gamma$ defined by $v$} is an induced subgraph spanned by $v$ and the vertices of a single connected components of $\Gamma_v$.
		\end{defin}
		
		If $v$ is a cut-vertex of a graph $\Gamma$ with blocks $\Gamma_1,\Gamma_2,\dots$, then $\Gamma$ is the connected sum of the blocks over the vertex $v$ itself, i.e.,  $\Gamma=\Gamma_1\underset{v}{\lor}\Gamma_2\underset{v}{\lor}\dots$
		
		\begin{fact}\label{fact:cv}
			Let $\Gamma'$ be a graph with a cut-vertex $v$. If $v$ is not central in $\Gamma'$, then there is a block $\Gamma_1$ defined by $v$ such that $v$ is neither a cut-vertex nor a central vertex of $\Gamma_1$. 
		\end{fact}
		\begin{proof}
			If $(\Gamma)_{i\in I}$ are the blocks of $\Gamma'$ defined by $v$, then, by definition, $v$ is not a cut-vertex for $\Gamma_i$, $i\in I$. If $v$ were central in $\Gamma_i$, for all $i\in I$, then $v$ would be central in $\Gamma'$.
		\end{proof}
		
			\begin{lem}\label{lem:bkbb}
			Let $\Gamma$ be a tree of Droms graphs. Then, $\mf B_\Gamma$ is a Droms object. 
			\begin{proof}
				We argue by induction on the number of vertices of $\Gamma$. 
				
				If $\Gamma$ is a (connected) Droms graph, then $\mf R_\Gamma$ is a Droms object by Theorem \ref{thm:bkraags}. Since $\Gamma$ is connected, either $\mf B_\Gamma$ is a standard subalgebra of $\li_\Gamma$, or it is a finitely generated subgroup of $G_\Gamma$. Hence, we deduce, by Theorem \ref{thm:bkraags}, that $\mf B_\Gamma$ is a Droms object. 
				
				If $\Gamma$ is not a Droms graph, then, there are two trees of Droms graphs $\Gamma_1$ and $\Gamma_2$ with a common vertex $v$ such that $\Gamma$ is the connected sum of the $\Gamma_i$'s along $v$. By induction, $\mf B_{\Gamma_i}$ are Droms objects, and hence so is their free product $\mf B_\Gamma=\mf B_{\Gamma_1}\amalg\mf B_{\Gamma_2}$.
			\end{proof}
		\end{lem}
		
		\subsection{The defining construction of trees of Droms graphs}
	We start by showing the following key result.
	\begin{lem}\label{lem:charBKgrph}
		Let $\Gamma$ be a connected chordal $(\bar H,\Delta)$-free graph. Then, $\Gamma$ has either a cut-vertex or a central vertex. 
		\begin{proof}
			We argue by induction on the number of vertices of $\Gamma$. If $\Gamma$ has at most two vertices, then any of them is trivially a central vertex. 
			Suppose that $\Gamma$ has at least $2$ vertices.
			
			Since $\Gamma$ is a chordal graph with no induced gem, it is a Ptolemaic graph, and hence we can use the defining construction of $\Gamma$ by means of leaves and twins. 
			
			\begin{case}
				$\Gamma$ has a leaf.
			\end{case}
			If $\{v,w\}$ is a leaf with $w$ of valency $1$, then $v$ is a cut-vertex for $\Gamma$. 
			\begin{case}
				$\Gamma$ has two (true) twins $v$ and $v'$. 
			\end{case}
			Let $\Gamma'$ be the induced subgraph spanned by all the vertices $\neq v'$. By induction, $\Gamma'$ has either a central vertex or a cut-vertex. 
			
			If $\Gamma'$ has a central vertex $z$, then $z$ is also a central vertex for $\Gamma$, because either $z=v$ or $z$ is adjacent to $v$, and hence, in both cases, $v'$ is adjacent to $z$.
			
			If $\Gamma'$ has no central vertices, then it must have a cut-vertex $w$. Two subcases might occur.
			\begin{enumerate}
				\item \textit{$v$ is not a cut-vertex for $\Gamma'$}. In this case, $w$ is a cut-vertex for $\Gamma$ as well.
				\item \textit{$v=w$ is a cut-vertex for $\Gamma'$}. Since $\Gamma'$ has no central vertex, by Fact \ref{fact:cv}, there is a block $\Gamma_1$ of $\Gamma$ such that $v$ is neither a central nor a cut-vertex for $\Gamma_1$. Let $\Gamma_2$ be the union of the blocks different from $\Gamma_1$. 
				By induction, $\Gamma_1$ contains either a central vertex or a cut-vertex. 
				\begin{enumerate}
					\item If $\Gamma_1$ has no central vertices, then, by induction, it has a cut-vertex $w'$. In particular, $v\neq w'$ and $w'$ is a cut-vertex for $\Gamma$. 
					\item Suppose that $\Gamma_1$ has a central vertex $w_1$, and let $w_2$ be a vertex in $\Gamma_2$ adjacent to $v$. 
					
					If $v$ is a leaf in $\Gamma_1$, then $w_1$ is a cut-vertex for $\Gamma_1$, and hence for $\Gamma$ as in case (2)(a). Now suppose $v$ is not a leaf of $\Gamma_1$, and let $X\neq \emptyset$ be the set of vertices of $\Gamma_1$ that are not adjacent to $v$; in particular, all $x\in X$ lie at distance $2$ from $v$. If there existed a vertex $w_1'\neq w_1$ that is adjacent to both $v$ and a vertex $x\in X$, then $\Gamma$ would have an induced overlapping gems graph $\bar H$ spanned by $\{v,v',w_1,w_1',x,w_2\}$. 
					It follows that all the vertices $w_1'\neq w_1$ of $\Gamma_1$ that are adjacent to $v$ are not adjacent to any vertex of $X$, and hence $w_1$ is a cut-vertex for $\Gamma$. 
					
				\end{enumerate}
			\end{enumerate}
			It remains to consider the last case.
			\begin{case}
				$\Gamma$ has two false twins $v$ and $v'$ (and the neighbourhood of $v$ is a complete graph). 
			\end{case}
			Let $\Gamma'$ be the induced subgraph spanned by all the vertices $\neq v'$. By induction, $\Gamma'$ has either a central vertex or a cut-vertex. 
			\begin{enumerate}
				\item Let $z$ be a central vertex of $\Gamma'$. If $v=z$, then $\Gamma'$ is a complete graph, and any vertex $w\notin\{v,v'\}$ is central in $\Gamma$. 
				
				\noindent If $v\neq z$, then $z$ is adjacent to $v'$ as well, and hence it is central in $\Gamma$. 
				
				\item If $\Gamma'$ has no central vertex, then it must have a cut-vertex $w$. 
				Since the neighbourhood of $v$ is a complete graph, $v$ cannot be a cut-vertex for $\Gamma'$, i.e., $v\neq w$. It follows that $w$ is a cut-vertex of $\Gamma$, and the proof is complete.\end{enumerate}
		\end{proof}
	\end{lem}
We can now characterize the trees of Droms graphs in terms of forbidden induced subgraphs.

		\begin{thm}
		Let $\Gamma$ be a finite simplicial graph. Then the following statements are equivalent:
		\begin{enumerate}
			\item $\Gamma$ is a connected chordal graph that does not contain any induced gem $\nabla$ nor overlapping-gems $\bar H$,
			\item $\Gamma$ is a tree of Droms graphs.
		\end{enumerate}
		\begin{proof}
			 If $\Gamma$ is a tree of Droms graphs, then $\mf B_\Gamma$ is a Droms object by Lemma \ref{lem:bkbb}. In particular, $\Gamma$ is a ptolemaic graph with no induced $\bar H$  by Examples \ref{ex:counter}. 
			 
			For the converse, suppose $\Gamma$ is a connected, chordal, $(\bar H,\Delta)$-free graph and argue by induction on the number of vertices. The $1$-vertex case is obvious.
			By Lemma \ref{lem:charBKgrph}, $\Gamma$ has either a central vertex or a cut-vertex. 
			If $\Gamma$ has a cut vertex $v$, then there are two connected subgraphs $\Gamma_1$ and $\Gamma_2$ with common vertex $v$ such that $\Gamma=\Gamma_1\underset{v}{\lor}\Gamma_2$. By induction, both $\Gamma_1$ and $\Gamma_2$ are trees of Droms graphs, and hence so is $\Gamma$. 
			
			If $\Gamma$ has no cut-vertex, let $z$ be a central vertex. Consider the induced subgraph $\Gamma'$ of $\Gamma$ spanned by all the vertices $\neq z$. If $\Gamma'$ is not a Droms graph, then it must contain an induced path of length $3$, and hence its cone $\Gamma$ contains an induced gem, contradicting the hypotheses. It follows that $\Gamma'$ is a Droms graph, and hence so is its cone $\Gamma$.
		\end{proof}
	\end{thm}
	
	We deduce Theorem A, analogous to Theorem \ref{thm:bkraags}.
		\begin{thm}\label{thm:bkbb}
		Let $\Gamma$ be a finite simplicial graph. Then, $\Gamma$ is a tree of Droms graph if, and only if, anyone of the following statements holds:
		\begin{enumerate}
			\item All subgroups of $B_\Gamma$ are RAAGs (possibly defined by infinite graphs),
			\item All standard subalgebras of $\mc B_\Gamma$ are RAAG Lie algebras,
			\item The Lie algebra $\mc B_\Gamma$ is BK,
			\item The Bestvina-Brady object $\mf B_\Gamma$ is finitely-generated and locally Droms,
			\item If $k$ is a field, the cohomology $k$-algebra $H^\bu(\mf R_\Gamma,k)$ is universally Koszul,
			\item The pro-$p$ group $ B_{\Gamma,p}$ is BK,
			\item Every closed subgroup of $ B_{\Gamma,p}$ is a pro-$p$ RAAG,
			\item There exists a field $\K$ containing a primitive $p^{th}$ root of $1$ such that $ B_{\Gamma,p}\simeq G_\K(p)$. 
		\end{enumerate}
	\end{thm}
	In particular, both the elementary type conjecture \cite{efrat} and the universal Koszulity conjecture \cite{MPPT} are true within the class of Bestvina-Brady pro-$p$ groups.
%
	\section{Applications of Theorem \ref{thm:bkbb}}
	\subsection{Surface groups} RAAGs can contain surface groups of high genus. For instance, by Servatius, Droms and Servatius \cite{servdroms}, if an $n$-cycle is induced in a graph $\Gamma$, then the derived subgroup $G_\Gamma'$ contains the fundamental group of a closed oriented surface of genus $g=1+(n-4)2^{n-3}$.

	We now prove that no oriented surface group of genus $\geq 2$ is a Bestvina-Brady group. 
	\begin{cor}
		Suppose that the Bestvina-Brady Lie algebra $\mc B_\Gamma$ is a quadratic $1$-relator Lie algebra over a field $k$. Then, $\mc B_\Gamma$ decomposes as the free product of an abelian Lie algebra of dimension $2$ and a free Lie algebra. 
		
		In particular, the fundamental group of an orientable closed surface of genus $\geq 2$ is not isomorphic to a Bestvina-Brady group.
	\end{cor}
	\begin{proof}
		By \cite{sb}, $\mc B=\mc B_\Gamma$ is BK, and hence $\Gamma$ is a tree of Droms graphs by Theorem \ref{thm:bkbb}. 
		
		Now, \cite[Lem. 2.2]{sb_kosz} provides the Betti numbers $b_i(\argu):=\dim H^i(\argu,k)$ of $\mc B$ in terms of those of $\li_\Gamma$: 
		\[1\overset{1-rel.}{=}b_2(\mc B)=b_2(\li_\Gamma)-b_1(\li_\Gamma)+1.\]
		Recall that the number of vertices (resp. of edges) of $\Gamma$ equals the Betti number $b_1(\li_\Gamma)$ (resp., $b_2(\li_\Gamma)$), so that $\abs{V}=\abs{E}$. 
		
		By definition, a (non-induced) tree $T=(V(T),E(T))$ of $\Gamma=(V,E)$ is a spanning tree if $V=V(T)$. As $\Gamma$ is connected, it admits a spanning tree $T$. One has $\abs{E(T)}-\abs{V(T)}+1=0$, and hence $1=\abs V-\abs{E(T)}=\abs E-\abs{E(T)}$, i.e., $\Gamma$ is obtained from $T$ by adding a single edge to a pair of its vertices, producing an induced cycle $C$. Since $\Gamma$ is a tree of Droms graphs, $C$ has length $3$ and $\mc B$ is the free product of $\mc B_C\simeq k^2$ and a free Lie algebra. 
		
		Now, let $G$ be the fundamental group of an orientable closed surface of genus $n$. Then, $G$ can be presented as \[G=\pres{x_1,y_1,\dots,x_n,y_n}{[x_1,y_1][x_2,y_2]\cdots[x_n,y_n]}\]
		
		The associated Lie algebra was computed by Labute \cite{labute}: \[\gr G=\pres{x_1,y_1,\dots,x_n,y_n}{[x_1,y_1]+[x_2,y_2]+\dots+[x_n,y_n]}\]
		
		If $G=B_\Gamma$ for some graph $\Gamma$, then, for any field $k$, $\gr G\otimes k$ is isomorphic with the Bestvina-Brady Lie $k$-algebra $\mc B_\Gamma$. By the first part of the proof, $n=1$. 
	\end{proof}
	The same proof shows that a Demu\v skin group \cite{demush} (see also \cite{MPPT} and \cite{labutedemush}) occurs as a Bestvina-Brady pro-$p$ group only if it is $2$-generated.
	
	Recall that a Lie $k$-algebra $\li$ is a Poincaré duality Lie algebra of dimension $n\in \mathbb N$ if it has cohomological dimension $n$, $H^n(\li,k)$ is $1$-dimensional, and the cup-product of the trivial coefficient cohomology ring defines non-degenerate pairing $H^i(\li,k)\otimes H^{n-i}(\li,k)\to H^n(\li,k)$. For instance, all the finite dimensional Lie algebras of dimension $n$ are Poincaré duality of dimension $n$, as well as any quadratic $1$-relator Lie algebra.
\begin{prop}
	Let $\Gamma$ be a finite simplicial graph such that $\mc B_\Gamma$ is a Poincaré duality Lie algebra of dimension $n$. Then, $\mc B_\Gamma$ is abelian. 
	\begin{proof} 
		Suppose $n\geq 1$. Since the cohomological dimension of $\mc B_\Gamma$ is $n$, the graph $\Gamma$ has a $(n+1)$-clique $\Delta$.
		
		Let $v$ be any vertex of $\Gamma$, and let $\Gamma_v$ be the induced subgraph spanned by the vertices $\neq v$. Then, $\mc B_{\Gamma_v}$ is a proper subalgebra of $\mc B_\Gamma$, and hence it has cohomological dimension $\leq n-1$ by \cite[p. 792]{sb}, proving that $\li_{\Gamma_v}$ has cohomological dimension $\leq n$. 
		
		In particular, $\Gamma_v$ contains no $(n+1)$-clique, that amount to saying $v\in \Delta$, and $\Gamma=\Delta$.
	\end{proof}
\end{prop}
The same argument applies to any cocyclic ideal of $\li_\Gamma$ in the place of $\mc B_\Gamma$.

\subsection{Bloch-Kato version of the $b_2$-conjecture}
The following question was raised by Weigel in \cite{weig}:
\begin{question}\label{quest:b2}
	Let $A$ be a Koszul algebra of cohomological dimension $d$ over a field $k$. Is it true that \[\dim(\ext^{2,2}_A(k,k))\leq \frac{d-1}{2d}\dim(\ext^{1,1}_A(k,k))^2?\]
\end{question}
This has positive answer in case the eigenvalues of $A$ are all real numbers (see \cite{sb_kosz}); in particular, this is true when $d=2$.

Aiming to answer this question in case $A$ is the universal enveloping algebra of a Lie algebra, the author introduced in \cite{sb_kosz} an invariant of a finitely presented graded Lie algebra $\li$ in terms of its low-degree Betti numbers as
\[\omega(\li):=(\cd(\li)-1)b_1(\li)^2-2\cd(\li)b_2(\li).\]

It was proved by Weigel \cite{weig} following Tur\'an \cite{turan} that, if $\li$ is a RAAG Lie algebra, then $\omega(\li)\geq 0$, giving a positive answer to Question \ref{quest:b2} within that class of Lie algebras. 


Now let $\mc B$ be a cocyclic ideal of type FP$_2$ of a Koszul Lie algebra $\li$. If $\li$ has cohomological dimension $d=n+1$, then a similar computation to that of \cite[Lem. 3.15]{sb_kosz} shows that \[(n+1)\omega(\mc B)=n\omega(\li)-(b_1(\li)-n-1)^2.\]
In particular, the invariant $\omega$ of $\li$ is related with that of any finitely presented cocyclic ideals of $\li$.

\begin{cor}
	Let $\Gamma$ be a tree of Droms graphs of clique number $n+1$. Then, 
	\[n\omega(\li_\Gamma)\geq (b_1(\li_\Gamma)-n-1)^2\]
	
	Hence, if $\Gamma$ has $v$ vertices and $e$ edges, then \[n(v^2-2e-2)\geq (v-1)^2.\]
\end{cor}
\begin{proof}
	Since $\mc B_\Gamma$ is a RAAG Lie algebra, the inequality $\omega(\mc B_\Gamma)\geq 0$ follows from Tur\`an's theorem (se also\cite[Sec. 3.3]{sb_kosz}).
\end{proof}
The following is a graph-theoretic reformulation of Question 2 of \cite{weig} for Bestvina-Brady Lie algebras which are Koszul.
\begin{question}\label{quest:b2bb}
	Let $\Gamma$ be a finite simplicial graph with acyclic (over an arbitrary field) flag complex of dimension $n$. Is it true that 
	\begin{equation}\label{eq:b2bb}
		n(v^2-2e-1)\geq (v-1)^2\end{equation}
	where $v$ and $e$ are the number of vertices and edges of $\Gamma$, respectively?
	
	Equivalently, is it true that finite acyclic flag complexes have dimension at least $\frac{(v-1)^2}{v^2-2e-1}$?
\end{question}

\begin{rem}
	The inequality has been confirmed for $v\leq 8$ by means of an easy, yet far from being optimal, code in SageMath.
\definecolor{codegreen}{rgb}{0,0.6,0}
\definecolor{codegray}{rgb}{0.5,0.5,0.5}
\definecolor{codepurple}{rgb}{0.58,0,0.82}
\definecolor{backcolour}{rgb}{0.95,0.95,0.92}
\lstdefinestyle{mystyle}{
	language=Python,
	commentstyle=\color{codegreen},
	keywordstyle=\color{magenta},
	numberstyle=\tiny\color{codegray},
	stringstyle=\color{codepurple},
	basicstyle=\ttfamily\footnotesize,
	breakatwhitespace=false,
	breaklines=true,
	captionpos=b,
	keepspaces=true,
	numbers=left,
	numbersep=5pt,
	showspaces=false,
	showstringspaces=false,
	showtabs=false,
	tabsize=2
}
\lstset{style=mystyle}

		\begin{lstlisting}
for vert in [1..8]:
	for G in graphs(vert):  
		if G.is_connected():
			D = G.clique_complex()
			if D.is_acyclic():
				edges = G.num_edges()
				n = D.dimension()
				om = n * (vert**2 - 2 * edges - 1) - (vert - 1)**2 
				if om < 0:
					print(vert, edges)
					G.show()\end{lstlisting}

%
%

\end{rem}

For $2$-dimensional acyclic flag complexes, we deduce an interesting upper bound for the number of edges. 
\begin{prop}
	Let $\Gamma$ be a graph with acyclic, $2$-dimensional flag complex $\Delta_\Gamma$. Then the inequality (\ref{eq:b2bb}) holds: \[(v+1)^2\geq 4(e+1).\]
\end{prop}
\begin{proof}
	Since $\mc B_\Gamma$ is Koszul and has cohomological dimension at most $2$, the result follows from the fact that Question \ref{quest:b2} has positive answer for $\mc B_\Gamma$ (\hspace{1pt}\cite{weig}).
\end{proof}

	\section{Coherence of Bestvina-Brady objects}
	By \cite{cmp2} and \cite{bb}, the object $\mf R_\Gamma$ is locally of type $FP_\infty$ over $k$ iff $\Gamma$ is chordal; hence the $k$-cohomology of $\mf B_\Gamma$ is a Koszul algebra when $\Gamma$ is connected and chordal. In turn, by the universal coefficient theorem, this gives an algebraic proof that the flag complex of a connected chordal graph is acyclic (see also \hspace{1pt}\cite{acyclicChordal}).
	
		The coherence property of $\mf B_\Gamma$ is equivalent to that of $\mf R_\Gamma$.
\begin{prop}\label{prop:cohe}
	Let $\Gamma$ be a finite simplicial graph. Then, the following statements are equivalent:
	\begin{enumerate}
		\item $\Gamma$ is a chordal graph,
		\item $\mf B_\Gamma$ is coherent,
		\item The derived sub-object $\mf B_\Gamma'=[\mf B_\Gamma,\mf B_\Gamma]$ is free.
	\end{enumerate} 
	\begin{proof}
		If $\Gamma$ is chordal, then the RAAG object $\mf R_\Gamma$ is coherent and with free derived sub-object (see \cite{dromschordal},\cite{servdroms},\cite{cmp2}), and hence the same holds for $\mf B_\Gamma$. This proves that $(1)$ implies both $(2)$ and $(3)$.
		
		Suppose now that $\Gamma$ contains an induced $n$-cycle $C$ for $n\geq 4$. Then, the sub-object $\mf B_C$ of $\mf B_\Gamma$ is finitely generated but not finitely presented by Theorem \ref{thm:BB_FPn}, proving that $\mf B_\Gamma$ is not coherent. Moreover, since $\mf B_C'=\mf R_C'$ is a non-free sub-object of $\mf B_\Gamma'$, the latter cannot be free.
	\end{proof}
\end{prop}

	 In the context of (pro-$p$) groups, it is proved that $B_T$ is a free (pro-$p$) group of finite rank when $T$ is a finite tree (\hspace{1sp}\cite{ilirPavel}). We now provide a proof for the Lie theoretic translation of that result. 
\begin{prop}\label{prop:freeBBLie}
	Let $\Gamma$ be a finite graph. Then, ${\mc B}_\Gamma$ is a free Lie algebra of rank $\abs{V(\Gamma)}-1$ if, and only if, $\Gamma$ is a tree.
\end{prop}
\begin{proof}
	First assume that $\Gamma=T$ is a tree. 
	We argue by induction on the number of vertices of $T$. 
	The result clearly holds when $T$ consists of a single vertex, for ${\mc B}_{\{v\}}=0$.
	Suppose that $T$ has at least $2$ vertices. 
	Since $T$ is a finite tree, it contains a leaf $v$ and we set $e=\{v,w\}\in E(T)$ for the unique edge containing $v$.
	If $T_0$ is the induced subtree of $T$ spanned by the vertices $\neq v$ of $T$, then ${\mc B}_{T_0}$ is free and $\li_T=\hnn_\phi(\li_{T_0},v)$, where $\phi:\gen{w}\to \li_{T_0}$ is the zero derivation. Since ${\mc B}_{T_0}={\mc B}_T\cap \li_{T_0}$ and $w\notin {\mc B}_{T_0}$, it follows that ${\mc B}_T$ is free by \cite[Thm. 4]{hnnLS}.
	
	For the converse, suppose $\mc B_\Gamma$ is a free Lie algebra of the prescribed finite rank, and $\Gamma$ is not a tree. Since $\mc B_\Gamma$ is finitely generated, $\Gamma$ is connected, and hence it must contain an induced $n$-cycle for some $n\geq 3$. 
	
	If $n=3$, then $\cd\li_\Gamma \geq 3$. Now, notice that $\li_\Gamma=\mc B_\Gamma\rtimes k$, and hence, by \cite[Lem. 2.2]{sb_kosz}, $\cd\mc B_\Gamma\geq \cd\li_\Gamma-1\geq 2$, proving that $\mc B_\Gamma$ is not free. 
	
	Assume now $n\geq 4$ and let $\Delta$ be an induced $n$-cycle in $\Gamma$. By Theorem \ref{thm:BB_FPn}, $\mc B_\Delta$ is a standard subalgebra of $\li_\Delta$ but it is not of type FP$_2$, and hence it is not a free Lie algebra. Since $\mc B_\Delta=\mc B_\Gamma\cap \li_\Delta$, we deduce that $\mc B_\Gamma$ contains a non-free subalgebra, proving that $\mc B_\Gamma$ is not free either.
\end{proof}
\begin{proof}[Geometric proof]
	Let $\Gamma$ be a finite connected graph with $n$ vertices and $m$ edges. It is a tree precisely when $n=m+1$. 
	
	If $\Gamma$ is a tree, then $\mc B_\Gamma$ is finitely presented by Theorem \ref{thm:BB_FPn}. Since $\li_\Gamma=\mc B_\Gamma\rtimes k$, it follows that $b_2(\mc B_\Gamma)=b_2(\li_\Gamma)-b_1(\mc B_\Gamma)=m-(n-1)=0$, and hence $\mc B_\Gamma$ is free by \cite[Thm. 5.2]{sb}.
	
	Conversely, if $\mc B_\Gamma$ is free of rank $n-1$, then it is of type FP and $m=b_2(\li_\Gamma)=b_2(\mc B_\Gamma)+b_1(\mc B_\Gamma)=0+n-1$, i.e., $\Gamma$ is a tree.
\end{proof}

\begin{cor}
	A finite graph $\Gamma$ is a tree if, and only if, $B_\Gamma$ is a free group of finite rank.
	\begin{proof}
		If $B_\Gamma(\Z)$ is a free group of finite rank, then $\Gamma$ is connected by Theorem \ref{thm:BB_FPn}, and $\mc B_\Gamma\simeq \gr B_\Gamma$ is free by \cite[Thm. 6.1]{serrelie} (see also \cite{magnus}). As a consequence of Proposition \ref{prop:freeBBLie}, we deduce that $\Gamma$ is a tree.
		
		Conversely, if $\Gamma$ is a tree, then $\mc B_\Gamma(\Z)$ is a free Lie algebra and $\gr B_\Gamma\simeq \mc B_\Gamma(\Z)$. It follows from the May spectral sequence that $\cd B_\Gamma\leq \cd\mc B_\Gamma(\Z)=1$, and hence $B_\Gamma$ is a free group by the Stallings-Swan theorem. 
	\end{proof}
\end{cor}

	\bibliography{mybibtex.bib}

\begin{bibdiv}
\begin{biblist}

\bib{ptol}{article}{
      author={Bandelt, H.J.},
      author={Mulder, H.M.},
       title={Distance-hereditary graphs},
        date={1986},
        ISSN={0095-8956},
     journal={Journal of Combinatorial Theory, Series B},
      volume={41},
      number={2},
       pages={182\ndash 208},
}

\bib{barth}{article}{
      author={Bartholdi, L.},
      author={H{\"a}rer, H.},
      author={Schick, Th.},
       title={Right-angled {Artin} groups and partial commutation, old and
  new},
        date={2020},
     journal={L’Enseignement Math{\'e}matique},
      volume={66},
      number={1},
       pages={33\ndash 61},
}

\bib{bb}{article}{
      author={Bestvina, M.},
      author={Brady, N.},
       title={Morse theory and finiteness properties of groups},
        date={1997},
     journal={Inventiones mathematicae},
      volume={129},
       pages={445\ndash 470},
}

\bib{sb}{article}{
      author={Blumer, S.},
       title={Kurosh theorem for certain {Koszul Lie} algebras},
        date={2023},
     journal={Journal of Algebra},
      volume={614},
       pages={780\ndash 805},
}

\bib{sb_kosz}{article}{
      author={Blumer, S.},
       title={Koszul {Lie} algebras and their subalgebras},
        date={2024},
     journal={arXiv:2412.08295},
}

\bib{bqw}{article}{
      author={Blumer, S.},
      author={Quadrelli, C.},
      author={Weigel, Th.},
       title={{Oriented Right-Angled Artin Pro-$\ell$ Groups and Maximal
  Pro-$\ell$ Galois Groups}},
        date={2023},
     journal={International Mathematics Research Notices},
}

\bib{cassquad}{article}{
      author={Cassella, A.},
      author={Quadrelli, C.},
       title={{Right-angled Artin groups and enhanced Koszul properties}},
        date={2021},
     journal={Journal of Group Theory},
      volume={24},
      number={1},
       pages={17\ndash 38},
}

\bib{demush}{article}{
      author={Demu{\v{s}}kin, S.~P.},
       title={The group of a maximal p-extension of a local field},
        date={1961},
     journal={Izv. Akad. Nauk SSSR Ser. Mat},
      volume={25},
       pages={329\ndash 346},
}

\bib{dromschordal}{article}{
      author={Droms, C.},
       title={Graph groups, coherence, and three-manifolds.},
        date={1987},
     journal={Journal of Algebra},
      volume={106},
      number={2},
       pages={484\ndash 489},
}

\bib{droms}{article}{
      author={Droms, C.},
       title={Subgroups of graph groups},
        date={1987},
     journal={Journal of Algebra},
      volume={110},
      number={2},
       pages={519\ndash 522},
}

\bib{efrat}{article}{
      author={Efrat, I.},
       title={{Orderings, valuations, and free products of Galois groups,
  Sem.}},
        date={1995},
     journal={Structure Alg{\'e}briques Ordonn{\'e}es, Univ. Paris VII},
      volume={54},
}

\bib{triviallyperfect}{article}{
      author={Golumbic, M.C.},
       title={Trivially perfect graphs},
        date={1978},
     journal={Discrete Mathematics},
      volume={24},
      number={1},
       pages={105\ndash 107},
}

\bib{kimroush}{article}{
      author={Kim, K.H.},
      author={Roush, F.W.},
       title={Homology of certain algebras defined by graphs},
        date={1980},
     journal={Journal of Pure and Applied Algebra},
      volume={17},
      number={2},
       pages={179\ndash 186},
}

\bib{cmp}{article}{
      author={Kochloukova, D.H.},
      author={Mart{\'\i}nez-P{\'e}rez, C.},
       title={{Bass-Serre theory for Lie algebras: A homological approach}},
        date={2021},
     journal={Journal of Algebra},
      volume={585},
       pages={143\ndash 175},
}

\bib{cmp2}{article}{
      author={Kochloukova, D.H.},
      author={Mart{\'\i}nez-P{\'e}rez, C.},
       title={{Coabelian Ideals in $\mathbb{N}$-graded Lie algebras and
  applications to right angled Artin Lie algebras}},
        date={2022},
     journal={Israel Journal of Mathematics},
      volume={247},
      number={2},
       pages={797\ndash 829},
}

\bib{labutedemush}{article}{
      author={Labute, J.P.},
       title={Classification of {Demushkin} groups},
        date={1967},
     journal={Canadian Journal of Mathematics},
      volume={19},
       pages={106\ndash 132},
}

\bib{labute}{article}{
      author={Labute, J.P.},
       title={The determination of the {Lie} algebra associated to the lower
  central series of a group},
        date={1985},
     journal={Transactions of the American Mathematical Society},
      volume={288},
      number={1},
       pages={51\ndash 57},
}

\bib{hnnLS}{article}{
      author={Lichtman, A.},
      author={Shirvani, M.},
       title={{HNN-extensions of Lie algebras}},
        date={1997},
     journal={Proceedings of the American Mathematical Society},
      volume={125},
      number={12},
       pages={3501\ndash 3508},
}

\bib{magnus}{book}{
      author={Magnus, W.},
      author={Karrass, A.},
      author={Solitar, D.},
       title={Combinatorial group theory: {Presentations} of groups in terms of
  generators and relations},
   publisher={Courier Corporation},
        date={2004},
}

\bib{may}{article}{
      author={May, J.P.},
       title={{The cohomology of restricted Lie algebras and of Hopf
  algebras}},
        date={1966},
     journal={Journal of Algebra},
      volume={3},
      number={2},
       pages={123\ndash 146},
}

\bib{MPPT}{article}{
      author={Min{\'a}{\v{c}}, J.},
      author={Palaisti, M.},
      author={Pasini, F.W.},
      author={T{\^a}n, N.D.},
       title={{Enhanced Koszul properties in Galois cohomology}},
        date={2020},
     journal={Research in the Mathematical Sciences},
      volume={7},
       pages={1\ndash 34},
}

\bib{psRAAG}{article}{
      author={Papadima, S.},
      author={Suciu, A.I.},
       title={Algebraic invariants for right-angled {Artin} groups},
        date={2006},
     journal={Mathematische Annalen},
      volume={334},
       pages={533\ndash 555},
}

\bib{psBB}{article}{
      author={Papadima, S.},
      author={Suciu, A.I.},
       title={Algebraic invariants for {Bestvina}--{Brady} groups},
        date={2007},
     journal={Journal of the London Mathematical Society},
      volume={76},
      number={2},
       pages={273\ndash 292},
}

\bib{acyclicChordal}{article}{
      author={Parks, A.},
       title={Clique complex homology: a combinatorial invariant for chordal
  graphs},
        date={2013},
     journal={International Journal of Sciences},
      volume={2},
}

\bib{pp}{book}{
      author={Polishchuk, A.},
      author={Positselskii, L.E.},
       title={Quadratic algebras},
   publisher={American Mathematical Soc.},
        date={2005},
      volume={37},
}

\bib{serrelie}{book}{
      author={Serre, J.-P.},
       title={Lie {Algebras} and {Lie} {Groups}: 1964 {Lectures} given at
  {Harvard} {University}},
   publisher={Springer Science \& Business Media},
        date={1992},
      volume={1500},
}

\bib{servdroms}{article}{
      author={Servatius, H.},
      author={Droms, C.},
      author={Servatius, B.},
       title={Surface subgroups of graph groups},
        date={1989},
     journal={Proceedings of the American Mathematical Society},
      volume={106},
      number={3},
       pages={573\ndash 578},
}

\bib{ilirPavel}{article}{
      author={Snopce, I.},
      author={Zalesskii, P.},
       title={Right-angled {Artin} pro-p groups},
        date={2022},
     journal={Bulletin of the London Mathematical Society},
      volume={54},
      number={5},
       pages={1904\ndash 1922},
}

\bib{turan}{article}{
      author={Tur{\'a}n, P.},
       title={{Eine Extremalaufgabe aus der Graphentheorie}},
        date={1941},
     journal={Mat. Fiz. Lapok},
      volume={48},
      number={436-452},
       pages={61},
}

\bib{weig}{article}{
      author={Weigel, Th.},
       title={Graded {Lie} algebras of type {FP}},
        date={2015},
     journal={Israel Journal of Mathematics},
      volume={205},
       pages={185\ndash 209},
}

\bib{quasithreshold}{article}{
      author={Yan, J.-H.},
      author={Chen, J.-J.},
      author={Chang, G.~J.},
      author={others},
       title={Quasi-threshold graphs},
        date={1996},
     journal={Discrete applied mathematics},
      volume={69},
      number={3},
       pages={247\ndash 255},
}

\end{biblist}
\end{bibdiv}
	
\end{document}